\documentclass[12pt]{amsart}

\usepackage{amscd,amssymb,amsopn,amsmath,amsthm,mathrsfs,graphics,amsfonts,enumerate,verbatim,calc
}

\usepackage[all]{xy}

\usepackage{color}

\usepackage[OT2,OT1]{fontenc}
\newcommand\cyr{%
\renewcommand\rmdefault{wncyr}%
\renewcommand\sfdefault{wncyss}%
\renewcommand\encodingdefault{OT2}%
\normalfont
\selectfont}
\DeclareTextFontCommand{\textcyr}{\cyr}

\usepackage{amssymb,amsmath}

\DeclareFontFamily{OT1}{rsfs}{}
\DeclareFontShape{OT1}{rsfs}{n}{it}{<-> rsfs10}{}
\DeclareMathAlphabet{\mathscr}{OT1}{rsfs}{n}{it}

\numberwithin{equation}{section}
\newtheorem{theorem}{Theorem}[section]
\newtheorem{lemma}[theorem]{Lemma}

\newtheorem{question}{Question}

\theoremstyle{definition}
\newtheorem{definition}[theorem]{Definition}
\newtheorem{remark}[theorem]{Remark}
\theoremstyle{remark}

\newtheorem{example}[theorem]{Example}
\newtheorem{acknowledgement}{Acknowledgement}

\newcommand{\fm}{\frak{m}}

\newcommand{\fq}{\frak{q}}
\newcommand{\fa}{\frak{a}}

\begin{document}
\title[Upper bound of multiplicity]{Upper bound of multiplicity in prime characteristic}

\author[Duong Thi Huong]{Duong Thi Huong}
\address{Department of Mathematics, Thang Long University, Hanoi, Vietnam}
\email{duonghuongtlu@gmail.com}

\author[Pham Hung Quy]{Pham Hung Quy}
\address{Department of Mathematics, FPT University, Hanoi, Vietnam}
\email{quyph@fe.edu.vn}

\thanks{
The second author is partially supported by a fund of Vietnam National Foundation for Science
and Technology Development (NAFOSTED) under grant number
101.04-2017.10.}


\begin{abstract} Let $(R, \fm)$ be a local ring of prime characteristic $p$ and of dimension $d$ with the embedding dimension $v$. Suppose the Frobenius test exponent for parameter ideals $Fte(R)$ of $R$ is finite, and let $Q = p^{Fte(R)}$. It is shown that 
$$e(R) \le Q^{v-d} \binom{v}{d}.$$
We also improve the bound for $F$-nilpotent rings. Our result extends the main results of Huneke and Watanabe \cite{HW15} and of Katzman and Zhang \cite{KZ18}.
\end{abstract}

\maketitle

\section{Introduction}
Throughout this paper, let $(R, \fm)$ be a Noetherian commutative local ring of prime characteristic $p>0$ and of dimension $d$. Our work is inspired by the work of Huneke and Watanabe \cite{HW15} in what they gave an upper bound of the multiplicity $e(R)$ of an $F$-pure ring $R$ in terms of the embedding dimension $v$. Namely, Huneke and Watanabe proved that 
$$e(R) \le \binom{v}{d}$$
for any $F$-pure ring. If $R$ is $F$-rational, the authors of \cite{HW15} provided a better bound that $e(R) \le \binom{v-1}{d-1}$ (cf. \cite[Theorem 3.1]{HW15}). Recently, Katzman and Zhang tried to remove the $F$-pure condition in Huneke-Watanabe's theorem by using the Hartshorne-Speiser-Lyubeznik number $HSL(R)$. Notice that $HLS(R) = 0$ if $R$ is $F$-injective (e.g. $R$ is $F$-pure). If $R$ is Cohen-Macaulay, Katzman and Zhang \cite[Theorem 3.1]{KZ18} proved the following inequality
$$e(R) \le Q^{v-d} \binom{v}{d},$$
where $Q = p^{HSL(R)}$. They also constructed examples to show that their bound is asymptotically sharp (cf. \cite[Remark 3.2]{KZ18}).\\

The key ingredient of this paper is the Frobenius test exponent for parameter ideals of $R$. Recall that the Frobenius test exponent for parameter ideals of $R$, denoted by $Fte(R)$, is the least integer (if exists) $e$ satisfying that $(\fq^F)^{[p^e]} = \fq^{[p^e]}$ for every parameter ideal $\fq$, where $\fq^F$ is the Frobenius closure of $\fq$. It is asked by Katzman and Sharp that whether $Fte(R) < \infty$ for every (equidimensional) local ring (cf. \cite{KS06}). If $R$ is Cohen-Macaulay then $Fte(R) = HSL(R)$. Moreover the question of Katzman and Sharp has affirmative answers when $R$ is either generalized Cohen-Macaulay by \cite{HKSY06} or $F$-nilpotent by \cite{Q19} (see the next section for the details). The main result of the present paper is as follows.
\begin{theorem}\label{MT} Let $(R, \fm)$ be a local ring of dimension $d$ with the embedding dimension $v$. Then
\begin{enumerate}
\item If $R$ is $F$-nilpotent then 
$$e(R) \le Q^{v-d}\binom{v-1}{d-1},$$ 
where $Q = p^{Fte(R)}$.
\item Suppose $Fte(R) < \infty$. Then 
$$e(R) \le Q^{v-d}\binom{v}{d},$$ 
where $Q = p^{Fte(R)}$.
\end{enumerate}

\end{theorem}
We will prove the above theorem in the last section. In the next section we collect some useful materials.
\section{Preliminaries}
\noindent {\bf $F$-singularities.} We firstly give the definition of the tight closure and the Frobenius closure of ideals.
\begin{definition}[\cite{HH90,H96}]
Let $R$ have characteristic $p$. We denote by $R^{\circ}$ the set of elements
of $R$ that are not contained in any minimal prime ideal.Then for any ideal $I$ of $R$ we define
\begin{enumerate}
  \item The {\it Frobenius closure} of $I$, $I^F = \{x \mid  x^Q \in I^{[Q]} \text{ for some } Q = p^e\}$, where $I^{[Q]} = (x^Q \mid x \in I)$.
  \item The {\it tight closure} of $I$, $I^* = \{x \mid cx^Q \in I^{[Q]} \text{ for some } c \in R^{\circ} \text{ and for all } Q = p^e \gg 0\}$.
\end{enumerate}
\end{definition}
We next recall some classes of $F$-singularities mentioned in this paper.
\begin{definition}
A local ring $(R, \fm)$ is called {\it $F$-rational} if it is a homomorphic image of a Cohen-Macaulay local ring and every parameter ideal is tight closed, i.e. $\fq^* = \fq$ for all $\fq$.
\end{definition}
\begin{definition}
A local ring $(R, \fm)$ is called {\it $F$-pure} if the Frobenius endomorphism $F: R \to R, x \mapsto x^p$ is a pure homomorphism. If $R$ is $F$-pure, then it is proved that every ideal $I$ of $R$ is Frobenius closed, i.e. $I^F = I$ for all $I$.
\end{definition}
The Frobenius endomorphism of $R$ induces the natural Frobenius action on local cohomology $F: H^i_{\fm}(R) \to H^i_{\fm}(R)$ for all $i \ge 0$. By a similar way, we can define the Frobenius closure and tight closure of zero submodule of local cohomology, and denote by $0^F_{H^i_{\fm}(R)}$ and $0^*_{H^i_{\fm}(R)}$ respectively. 
\begin{definition}
\begin{enumerate}
\item A local ring $(R, \fm)$ is called {\it $F$-injective} if the Frobenius action on $H^i_{\fm}(R)$ is injective, i.e. $0^F_{H^i_{\fm}(R)} = 0$, for all $i \ge 0$.
\item  A local ring $(R, \fm)$ is called {\it $F$-nilpotent} if the Frobenius actions on all lower local cohomologies $H^i_{\fm}(R)$, $i \le d-1$, and $0^*_{H^d_{\fm}(R)}$ are nilpotent, i.e. $0^F_{H^i_{\fm}(R)} = H^i_{\fm}(R)$ for all $i \le d-1$ and $0^F_{H^d_{\fm}(R)} = 0^*_{H^d_{\fm}(R)}$.
\end{enumerate}
\end{definition}
\begin{remark}\label{F-sing}
\begin{enumerate}
\item It is well known that an equidimensional local ring $R$ is $F$-rational if and only if it is Cohen-Macaulay and $0^*_{H^d_{\fm}(R)} = 0$.
\item An excellent equidimensional local ring is $F$-rational if and only if it is both $F$-injective and $F$-nilpotent.
\item Suppose every parameter ideal of $R$ is Frobenius closed. Then $R$ is $F$-injective (cf. \cite[Main Theorem A]{QS17}). In particular, an $F$-pure ring is $F$-injective.
\item An excellent equidimensional local ring $R$ is $F$-nilpotent if and only if $\fq^* = \fq^F$ for every parameter ideal $\fq$ (cf. \cite[Theorem A]{PQ18}). 
\end{enumerate}
\end{remark}
\noindent {\bf $F$-invariants.} We will bound the multiplicity of a local ring of prime characteristic in terms of the Frobenius test exponent for parameter ideals of $R$. Let $I$ be an ideal of $R$. The {\it Frobenius test exponent} of $I$, denoted by $Fte(I)$, is the smallest number $e$ satisfying that $(I^F)^{[p^e]} = I^{[p^e]}$. By the Noetherianess of $R$, $Fte(I)$ exists (and depends on $I$). In general, there is no upper bound for the Frobenius test exponents of all ideals in a local ring by the example of Brenner \cite{B06}. In contrast, Katzman and Sharp \cite{KS06} showed the existence of a uniform bound of Frobenius test exponents if we restrict to the class of parameter ideals in a Cohen-Macaulay local ring. For any local ring $(R, \fm)$ of prime characteristic $p$ we define the {\it Frobenius test exponent for parameter ideals}, denoted by $Fte(R)$, is the smallest integer $e$ such that $(\fq^F)^{[p^e]} = \fq^{[p^e]}$ for every parameter ideal $\fq$ of $R$, and $Fte(R) = \infty$ if we have no such integer. Katzman and Sharp raised the following question.
\begin{question}\label{Question}
Is $Fte(R)$ a finite number for any (equidimensional) local ring?
\end{question}
The Frobenius test exponent for parameter ideals is closely related to an invariant defined by the Frobenius actions on the local cohomology modules $H^i_{\fm}(R)$, namely {\it the Hartshorne-Speiser-Lyubeznik number} of $H^i_{\fm}(R)$. The Hartshorne-Speiser-Lyubeznik number of $H^i_{\fm}(R)$ is a nilpotency index of Frobenius action on $H^i_{\fm}(R)$ and it is defined as follows 
$$HSL(H^i_{\fm}(R)) = \min \{e \mid F^e(0^F_{H^i_{\fm}(R)}) = 0  \}.$$
By \cite[Proposition 1.11]{HS77} and \cite[Proposition 4.4]{L97} $HSL(H^i_{\fm}(R))$ is well defined (see also \cite{Sh07}). The Hartshorne-Speiser-Lyubeznik number of $R$ is $HSL(R) = \max \{HSL(H^i_{\fm}(R)) \mid i = 0, \ldots, d\}$.
\begin{remark}\label{Fte}
\begin{enumerate}
\item If $R$ is Cohen-Macaulay then $Fte(R) = HSL(R)$ by Katzman and Sharp \cite{KS06}. In general, the authors of this paper proved in \cite{HQ19} that $Fte(R) \ge HSL(R)$. Moreover, Shimomoto and the second author \cite[Main Theorem B]{QS17} constructed a local ring satisfying that $HSL(R) = 0$, i.e. $R$ is $F$-injective, but $Fte(R) > 0$.
\item Huneke, Katzman, Sharp and Yao \cite{HKSY06} gave an affirmative answer for Question \ref{Question} for generalized Cohen-Macaulay rings. 
\item Recently, the second author \cite{Q19} provided a simple proof for the theorem of Huneke, Katzman, Sharp and Yao. By the same method he also proved that $Fte(R) < \infty$ if $R$ is $F$-nilpotent. Very recently, Maddox \cite{M18} extended this result for {\it generalized $F$-nilpotent} rings.
\end{enumerate}

\end{remark}

\section{Proof of the main result}
This section is devoted to prove the main result of this paper. Without loss of generality we will assume that $R$ is complete with an infinite residue field. We need the following key lemma.
\begin{lemma}\label{briancon} Let $(R, \fm)$ be a local ring of dimension $d$, and $\fq$ a parameter ideal. 
\begin{enumerate}
\item If $R$ is $F$-nilpotent then $\overline{\fq^d} \subseteq \fq^F$, where $\overline{I}$ is the integral closure of ideal $I$.
\item In general we have $\overline{\fq^{d+1}} \subseteq \fq^F$.
\end{enumerate}
\end{lemma}

\begin{proof} (1) By the Brian\c{c}on-Skoda type theorem \cite[Theorem 5.6]{HH90} we have $\overline{\fq^d} \subseteq \fq^*$. The assertion now follows from Remark \ref{F-sing}(4).\\
(2) The assertion follows from \cite[Theorem 2.2]{KZ18}\footnote{In fact Katzman and Zhang \cite[Theorem 2.2]{KZ18} needed to assume that every $c \in R^{\circ}$ is a non-zero divisor, i.e. $R$ has no embedded primes. However, we can easily remove this condition by passing to the quotient ring $R/\fa$, where $\fa$ is the intersection of primary ideals corresponding to minimal primes in a primary decomposition of the zero ideal.}.
\end{proof}
We prove the main result of this paper.
\begin{theorem}\label{T3.2} Let $(R, \fm)$ be a local ring of dimension $d$ with the embedding dimension $v$. Then
\begin{enumerate}
\item If $R$ is $F$-nilpotent then 
$$e(R) \le Q^{v-d}\binom{v-1}{d-1},$$ 
where $Q = p^{Fte(R)}$.
\item Suppose $Fte(R) < \infty$ (e.g. $R$ is generalized Cohen-Macaulay or generalized $F$-nilpotent). Then 
$$e(R) \le Q^{v-d}\binom{v}{d},$$ 
where $Q = p^{Fte(R)}$.
\end{enumerate}

\end{theorem}
\begin{proof} Because the proofs of two assertions are almost the same, we will only prove (1). Since $R$ is $F$-nilpotent we have $Fte(R) < \infty$ by Remark \ref{Fte}(3). Let $\fq = (x_1, \ldots, x_d)$ be a minimal reduction of $\fm$. By Lemma \ref{briancon}(1) we have $\fm^d \subseteq \overline{\fm^d} = \overline{\fq^d} \subseteq \fq^F$. On the other hand we have $(\fq^F)^{[Q]} = \fq^{[Q]}$ by the definition of $Fte(R)$. Thus $(\fm^d)^{[Q]} \subseteq \fq^{[Q]}$. Extend $x_1, \ldots, x_d$ to a minimal set of generators $x_1, \ldots, x_d, y_1, \ldots, y_{v-d}$ of $\fm$. Now $R/\fq^{[Q]}$ is spanned by monomials
$$x_1^{\alpha_1}\cdots x_d^{\alpha_d} y_1^{\beta_1Q + \gamma_1} \cdots y_1^{\beta_{v-d}Q + \gamma_{v-d}},$$  
where $0 \le \alpha_1, \ldots, \alpha_d, \gamma_1, \ldots, \gamma_{v-d} < Q$ and $0 \le \beta_1 + \cdots + \beta_{v-d} < d$. The number of such monomials is $Q^v \binom{v-1}{d-1}$ so $\ell_R(R/\fq^{[Q]}) \le Q^v \binom{v-1}{d-1}$.\\
Since $\fq$ is a parameter ideal we have $e(\fq^{[Q]}) = Q^d e(\fq) = Q^d e(R)$ and $e(\fq^{[Q]}) \le \ell_R(R/\fq^{[Q]})$. Hence
$$e(R) = \frac{1}{Q^d} e(\fq^{[Q]}) \le \frac{1}{Q^d} \ell_R(R/\fq^{[Q]}) \le  \frac{1}{Q^d} Q^v \binom{v-1}{d-1} = Q^{v-d} \binom{v-1}{d-1}.$$
The proof is complete.
\end{proof}
Finally, we present an example to prove that we can not remove $Q = p^e$ in the previous theorem.
\begin{example} Let $R = \mathbb{F}_p[[X^4, X^3Y, XY^3, Y^4]]$. It is easy to see that $\dim R = 2$ and $e(R) = 4$. Moreover, we can check $H^1_{\frak m}(R) \cong \mathbb{F}_p$ and the Frobenius action on $H^1_{\frak m}(R)$ is nilpotent. Thus $HSL(H^1_{\frak m}(R)) = 1$. Let $S = \mathbb{F}_p[[X^4, X^3Y,X^2Y^2, XY^3, Y^4]]$, the integral closure of $R$. We have that $S$ is $F$-regular and $H^2_{\frak m}(R) \cong H^2_{\frak m}(S)$. Therefore $0^*_{H^2_{\frak m}(R)} = 0$ and $R$ is $F$-nilpotent. We have $Fte(R) \le 2$ by the main theorem of \cite{Q19}\footnote{We believe that $Fte(R) = 1$.}. We can not omit $Q^{v-d}$ in Theorem \ref{T3.2} (1) since 
$$e(R) = 4 > 3 = \binom{v-1}{d-1}.$$ 

\end{example}
\begin{acknowledgement} The authors are grateful to the referee for careful reading and valuable comments. 
\end{acknowledgement}


\begin{thebibliography}{99}

\bibitem{B06} H. Brenner, \emph{Bounds for test exponents}, Compos. Math. {\bf} 142 (2006), 451--463.

%\bibitem{BS98} M. Brodmann and R.Y. Sharp, \emph{Local cohomology: an algebraic introduction with geometric applications}, Cambridge University Press, 1998.

\bibitem{HS77} R. Hartshorne and R. Speiser, \emph{Local cohomological dimension in characteristicp}, Ann. of Math. {\bf 105} (1977), 45--79.

\bibitem{HH90}
M. Hochster and C. Huneke, \emph{Tight Closure, Invariant Theory, and the Brian\c{c}on-Skoda Theorem}, J. Amer. Math. Soc. {\bf3} (1990), 31--116.

\bibitem{H96}
C. Huneke,  \emph{Tight closure and its applications}, CBMS Lecture Notes in Mathematics,  Vol.\textbf{88}, Amer. Math. Soc., Providence, (1996).

\bibitem{HKSY06} C. Huneke, M. Katzman, R.Y. Sharp and Y. Yao, \emph{Frobenius test exponents for parameter ideals in generalized Cohen-Macaulay local rings},
J. Algebra {\bf 305} (2006), 516--539.

\bibitem{HW15} C. Huneke and K.-i. Watanabe, \emph{Upper bound of multiplicity of $F$-pure rings}, Proc. Amer. Math. Soc. {\bf 143} (2015), 5021--5026.

\bibitem{HQ19} D.T. Huong and P.H. Quy, \emph{Notes on the Frobenius test exponents}, Comm. Algebra, to appear.

\bibitem{KS06} M. Katzman and R.Y. Sharp, \emph{Uniform behaviour of the Frobenius closures of ideals generated
by regular sequences}, J. Algebra {\bf 295} (2006) 231--246.

\bibitem{KZ18} M. Katzman and W. Zhang, \emph{Multiplicity bounds in prime characteristic}, Comm. Algebra, to appear.

\bibitem{L97} G. Lyubeznik, \emph{$F$-modules: applications to local cohomology and $D$-modules in characteristic $p>0$}, J. reine
angew. Math. {\bf 491} (1997), 65--130.

\bibitem{M18} K. Maddox, \emph{A sufficient condition for finiteness of Frobenius test exponents}, preprint, arXiv:1809.10063.
 
\bibitem{PQ18} T. Polstra and P.H. Quy, \emph{Nilpotence of Frobenius actions on local cohomology and Frobenius closure of ideals}, preprint, arXiv:1803.04081.

\bibitem{QS17} P.H. Quy and K. Shimomoto, \emph{$F$-injectivity and Frobenius closure of ideals in Noetherian rings of characteristic $p> 0$}, Adv. Math. {\bf 313} (2017), 127--166.

\bibitem{Q19} P.H. Quy, \emph{On the uniform bound of Frobenius test exponents}, J. Algebra {\bf 518} (2019), 119--128.

\bibitem{Sh07} R.Y. Sharp, \emph{On the Hartshorne-Speiser-Lyubeznik theorem about Artinian modules with a Frobenius action}, Proc. Amer. Math. Soc. {\bf 135} (2007), 665--670.

\end{thebibliography}
\end{document}